\documentclass[a4paper, 12pt]{article}
\usepackage{amsmath,amssymb,amsthm}
\usepackage[scaled]{helvet}
\usepackage{braket}
\usepackage{cleveref}
\usepackage{mathrsfs}
\usepackage{amsmath}
\usepackage{amsfonts}
\allowdisplaybreaks[4]
\usepackage{amsthm}
\usepackage{amscd}
\usepackage[all]{xy}
\usepackage{comment}
\usepackage{graphicx}
\usepackage[top=30truemm,bottom=30truemm,left=25truemm,right=25truemm]{geometry}

 \makeatletter
    
    \@addtoreset{equation}{section}
  \makeatother

\theoremstyle{plain}

\newtheorem{thm}{Theorem}[section]
\newtheorem{lem}[thm]{Lemma}
\newtheorem{prop}[thm]{Proposition}

\newtheorem{exa}[thm]{Example}
\newtheorem{defn}[thm]{Definition}
\newtheorem{rem}[thm]{Remark}
\newtheorem{psythm}{Psychological Theorem}

\DeclareMathOperator{\rank}{rank}

\DeclareMathOperator{\Ker}{Ker}
\DeclareMathOperator{\Hom}{Hom}
\DeclareMathOperator{\Ext}{Ext}
\DeclareMathOperator{\Homo}{H}
\DeclareMathOperator{\Int}{Int}
\DeclareMathOperator{\Span}{span}
\newcommand{\barsigma}{\overline{\sigma}}

\DeclareMathOperator{\diag}{diag}

\DeclareMathOperator{\vol}{vol}
\newcommand{\R}{\mathbb{R}}

\newcommand{\C}{\mathbb{C}}
\newcommand{\A}{\mathbb{A}}
\newcommand{\Z}{\mathbb{Z}}
\newcommand{\LL}{\mathbb{L}}
\newcommand{\T}{\mathbb{T}}

\newcommand{\Gm}{\mathbb{G}_m}
\newcommand{\s}{\sigma}
\newcommand{\bs}{\barsigma}
\newcommand{\tk}{\tilde{\bf k}}
\newcommand{\ii}{\sqrt{-1}}

\crefname{thm}{Theorem}{Theorems}
\crefname{bthm}{Theorem}{Theorems}
\crefname{bprop}{Proposition}{Propositions}
\crefname{blem}{Lemma}{Lemmata}

\title{On Mellin-Barnes integral representations for GKZ hypergeometric functions}
\author{Saiei-Jaeyeong Matsubara-Heo\footnote{The author is financially supported by  JSPS KAKENHI Grant Number 17J03916}
\\ {\small (Graduate School of Mathematical Sciences, The University of Tokyo)}}

\begin{document}

\date{}
\maketitle

%\tableofcontents      %optional
\begin{abstract}      %optional
We consider Mellin-Barnes integral representations of GKZ hypergeometric equations. We construct integration contours in an explicit way and show that suitable analytic continuations give rise to a basis of solutions.
\end{abstract}

\begin{comment}
\underline{index}
\begin{enumerate}
\item Motivations from the theory of integral representations
\item From Laplace type integral to Mellin-Barnes integral
\item Separation of variables and Mellin-Barnes integrals for ``diamond'' cases
\item Construction of rapid decay homology basis for purely irregular cases
\item Mellin-Barnes integral satisfies GKZ 
\item Mellin-Barnes integral is convergent
\item Mellin-Barnes integral produces series solutions
%\item Borel transform of divergent series
\end{enumerate}
\end{comment}

\section{Introduction}

\indent

In this short paper, we discuss Mellin-Barnes integral representations of GKZ hypergeometric functions as well as its relation to Laplace integral representations. We  remember some basic notation. Let $n<N$ be positive integers, $c\in\mathbb{C}^{n\times 1}$ be a fixed parameter, and let $\{ {\bf a}(1),\dots,{\bf a}(N)\}\subset\mathbb{Z}^{N\times 1}$ be lattice points. We set $A=({\bf a}(1)\mid \cdots\mid {\bf a}(N))=(a_{ij}).$ Throughout this paper, we assume that $\Z A=\displaystyle\sum_{j=1}^N\Z {\bf a}(j)=\Z^{n\times 1}.$
The GKZ hypergeometric system is given by the family of equations 
\begin{equation}
M_A(c):
\left\{
\begin{array}{llll}
E_i\cdot f(z)&=&0 &(i=1,\cdots, n)\\
\Box_u\cdot f(z)&=&0& (u\in L_A=\Ker_{\mathbb{Z}}A)
\end{array}
\right.
\end{equation}
where

\begin{equation}
E_i=\sum_{j=1}^Na_{ij}z_j\frac{\partial}{\partial z_j}+c_i,\;\;
\Box_u=\prod_{u_j>0}\left(\frac{\partial}{\partial z_j}\right)^{u_j}-\prod_{u_j<0}\left(\frac{\partial}{\partial z_j}\right)^{-u_j}.
\end{equation}

\noindent
It is well-known that GKZ system $M_A(c)$ is holonomic (\cite{A}) and in particular, has a finite rank. For a non-resonant parameter $c\in\C^{n\times 1}$, one also has $ \rank M_A(c)=\vol_\Z(\Delta_A)$ (\cite{A}), where $\Delta_A$ is a convex polytope called Newton polytope and defined by the formula $\Delta_A=\text{c.h.}\{ 0,{\bf a}(1),\dots,{\bf a}(N)\}$ and $\vol_\Z(\Delta_A)$ is the Euclidian volume of $\Delta_A$ multiplied by $n!$ so that the standard simplex has a volume $1$.  GKZ system $M_A(c)$ naturally has a formal series solution $\phi_v(z)$ associated to any vector $v\in\C^{N\times 1}$ such that $Av=-c$ and defined by the formula
\begin{equation}
\varphi_v(z)=\sum_{u\in L_A}\frac{z^{u+v}}{\Gamma({\bf 1}+u+v)},
\end{equation}
where $L_A=\Ker\left( A\times:\Z^{N\times 1}\rightarrow\Z^{n\times 1}\right)$. This formal series $\varphi_v(z)$ is called $\Gamma$-series in the literature. Various classical hypergeometric series such as Thomae's generalised hypergeometric series, Appell-Lauricella series, Horn's series are recovered in terms of $\Gamma$-series by taking  suitable matrices $A$ and vectors $c$ and $v$ (\cite{GGR}, \cite{SST}).  Moreover, one can construct a basis of solutions consisting only of $\Gamma$-series with the aid of regular triangulation of $\Delta_A$ (\cite{FF}, \cite{GGR}).

On the other hand, classical hypergeometric functions have another aspect: integral representations. For example, Gauss hypergeometric function ${}_2F_1(\alpha,\beta,\gamma ;z)$ has the so-called Mellin-Barnes integral representation:
\begin{equation}
{}_2F_1(\alpha,\beta,\gamma;z)=\frac{\Gamma(\gamma)}{\Gamma(\alpha)\Gamma(\beta)}\int_C\frac{\Gamma(\alpha-s)\Gamma(\beta-s)}{\Gamma(\gamma-s)}(e^{\pi\sqrt{-1}}z)^{-s}ds
\end{equation}
where the integration contour $C$ is taken to be the so-called Hankel contour starting from $-\infty-\ii$, extending to $-\delta-\ii$, going around the origin in the positive direction, and going back to $-\infty-\ii$. It is sometimes more convenient to take $C$ as the Barnes contour $C=\delta+\ii (-\infty,\infty)$ $(\delta>0)$. On the other hand, rewriting the well-known Euler integral representation, we can obtain Laplace integral representation

\begin{align}
& \quad{}_2F_1(\alpha, \beta, \gamma ; z)\nonumber\\
 &=\frac{\pi\Gamma(\gamma)}{\sin\pi(\gamma-\alpha)\Gamma(\alpha)\Gamma(\beta)}\int_0^\infty\int_0^\infty\int^1_0 e^{ -(1-t_1)t_2-(1-zt_1)t_3} t_1^{\alpha-1}t_2^{\alpha-\gamma}t_3^{\beta-1}dt_1dt_2dt_3\;\; (|z|<1).
\end{align}
These integral representations are also generalised to GKZ hypergeometric functions. Mellin-Barnes integral for GKZ hypergeometric functions was introduced in \cite{B} and applied to computations of subgroups of monodromy groups. To our knowledge, \cite{B} is the only article that treats Mellin-Barnes integral for GKZ hypergeometric functions. On the other hand, several authors developed a systematic study of Laplace integral representations (\cite{ET}, \cite{SW}). In particular, a basis of cycles for Laplace integrals was constructed in \cite{ET} by the method of steepest descent under a restrictive assumption on the Newton polytope $\Delta_A$ and non-degeneracy assumption on the parameter $c$.

Based on these studies, we consider the following problem in this paper:

\begin{description}
\item[Problem]\mbox{}\\
 Construct a basis of solutions by means of Mellin-Barnes or Laplace integral representations so that it can be related to $\Gamma$-series.
\end{description}

\noindent
This problem is certainly basic, but it seems that the general description has been neglected. We solve the problem for Laplace integral representations when the Newton polyhedron has the shape of a diamond. We construct a standard basis of cycles via K\"unneth type argument. We will find that integration on these cycles naturally gives rise to  Mellin-Barnes integral representations. In the latter half, we introduce Mellin-Barnes integral representations for GKZ hypergeometric functions which is slightly different from the one treated in \cite{B}. After examining that Mellin-Barnes integral representations satisfy GKZ system $M_A(c)$ and that it is convergent, we relate it to series representations by computing residues of the integrand. Finally, we show that a suitable analytic continuation gives rise to a basis of solutions (\cref{thm:MainTheoremForMellin-Barnes}). We will find that the transformation matrix between Mellin-Barnes integral and series solutions associated to a regular triangulation $T$ is given in terms a character matrix of a finite Abelian group $G_\s$ associated to each simplex $\s\in T$.

We will discuss the construction of integration cycles for Laplace integrals without any restriction on the Newton polytope in the forthcoming paper. We were inspired by \cite{CG}, \cite{FF} and \cite{KP}.

\section{Construction of integration cycles for Laplace integral representations when the Newton polytope has a diamond shape}

\indent
In this section, we briefly collect known results on Laplace integral representations of GKZ hypergeometric functions and construct a standard basis of cycles when the Newton polytope has the shape of a diamond. We say that a parameter vector $c\in\C^{n\times 1}$ is non-resonant if $c\notin\Z+\Span_\C\{\Gamma\}$ for any face $\Gamma$ of $\Delta_A$ such that $0\in\Gamma.$ The following definition is standard in the literature.

\begin{defn}(\cite{A})
For any $z\in\A^N,$ we say that a Laurent polynomial $h_z(x)=\displaystyle\sum_{j=1}^Nz_jx^{{\bf a}(j)}$ is Newton non-degenerate if for any face $\Gamma$ of $\Delta_A$ such that $0\notin\Gamma$, one has 
\begin{equation}
\left\{x\in(\Gm)^n\mid d_xh^\Gamma_z(x)=0\right\}=\varnothing.
\end{equation}
Here, $h_z^\Gamma(x)=\displaystyle\sum_{{\bf a}(j)\in\Gamma}z_jx^{{\bf a}(j)}$. 
\end{defn}
\noindent
We put $\Omega=\left\{z\in\A^N\mid h_z(x)\text{ is Newton non-degenerate}\right\}$. The set $\Omega$ is called the Newton non-degenerate locus of $A.$ The following result is of fundamental importance.

\begin{thm}\label{theorem:fundamental}(Adolphson\cite{A}, %Esterov-Takeuchi\cite{ET},
% M. Hien\cite{H},
 Schulze-Walther\cite{SW})
\begin{enumerate}
\item $M_A(c)$ is holonomic and $M_A(c)$ is a connection on $\Omega.$
\item If $c$ is non-resonant, $\rank M_A(c)=vol_\mathbb{Z}\Delta_A.$
\item Let $\pi:(\Gm)^n_x\times\mathbb{A}^N_z\rightarrow\mathbb{A}^N_z$ be a natural projection and suppose $c$ is non-resonant. One has a natural isomorphism of $D_{\A^N}$-modules
\begin{equation}
M_A(c)\simeq\int_\pi\mathcal{O}_{(\Gm)^n_x\times\mathbb{A}^N_z}e^{h_z(x)}x^{\bf c}.
\end{equation}
\end{enumerate}
\end{thm}

\noindent
\cref{theorem:fundamental} (3) suggests that general solutions $f(z)$ of $M_A(c)$ should have the form 
\begin{equation}\label{LaplaceIntegral}
f(z)=\int e^{h_z(x)}x^{c-1}dx.
\end{equation}
We call (\ref{LaplaceIntegral}) Laplace integral representation. Now, let us describe the space of local solutions at $z\in\Omega.$ Throughout this paper, we assume that $c$ is non-resonant. By Cauchy-Kowalevsky-Kashiwara theorem and the commutativity of inverse image and analytification, we have
\begin{eqnarray}
\R\mathcal{H}om_{\mathcal{D}_{\C^N}}(\left(M_A(c)\right)^{an},\mathcal{O}_{\C^N})_z&\simeq&\R \Hom_{\C}(\left(\LL\iota_z^* M_A(c)\right)^{an},\C) \nonumber\\
 &\simeq&\R \Hom_{\C}(\left(\LL\iota_z^* M_A(c)\right),\C) \label{isom1}.
\end{eqnarray}
On the other hand, by the definition of integration of a $D$-module and projection formula, for any $z\in\A^N$ and inclusion $\iota_z:\{z\}\hookrightarrow\A^N,$ we have
\begin{equation}\label{deRham}
\LL\iota_z^*\int_\pi\mathcal{O}_{\mathbb{T}^n_x\times\mathbb{A}^N_z}e^{h_z(x)}x^{\bf c}\simeq(\Omega^{\bullet+n}((\Gm)^n),\nabla_z),
\end{equation}
where differential $\nabla_z$ is given by
\begin{equation}
\nabla_z=d+dh_z\wedge+\sum_{i=1}^nc_i\frac{dx_i}{x_i}\wedge.
\end{equation}
By the general result of M. Hien (\cite{H}), we see that the dual space of p-th cohomology group of this complex is given in terms of rapid decay cycles:
\begin{equation}\label{RDduality}
\int: \Homo_{p+n}^{r.d.}\left(\T^n,(\mathcal{O}_{(\Gm)^n},\nabla_z^\vee)\right)\overset{\sim}{\rightarrow}\left(\Homo^p\left(\Omega^{\bullet+n}((\Gm)^n),\nabla_z\right)\right)^\vee.
\end{equation}
For readers' convenience, we briefly explain the construction of rapid decay homology groups $\Homo^{r. d.}$ following \cite{H}. For each fixed $z$, we compactify $\T^n_x$ to a smooth projective variety $X$ so that $h_z(x)$ can be prolonged to a meromorphic function, i.e., so that we have a commutative diagram
\begin{equation*}
\xymatrix{
& \T^n_x \ar[r]^{h_z} \ar[d]_{\rm inclusion}&\C \ar[d]^{\rm inclusion}\\
& X   \ar[r]^{\tilde{h}_z}                             &\mathbb{P}^1.}
\end{equation*}
Then, setting $D=X\setminus\T^n_x$, we have a decomposition $D=D^{irr}\cup D^{\prime}$, where $D^{\prime}=\tilde{h}_z^{-1}(\infty)$. We assume that $D$ is normal crossing. Then we consider the real oriented blow-up $\tilde{\mathbb{P}}^1=\C\cup S^1\infty$ of a projective line and $\tilde{X}$ along $D$. Here, $S^1$ is the unit circle $S^1=\{ z\in\C\mid |z|=1\}$. Then we have a commutative diagram 
\begin{equation*}
\xymatrix{
& \tilde{X} \ar[r]^{\tilde{h}_z} \ar[d]_{\pi}&\tilde{\mathbb{P}}^1 \ar[d]^{\pi}\\
& X   \ar[r]^{\tilde{h}_z}                             &\mathbb{P}^1,}
\end{equation*}
where $\pi$ is a canonical morphism of the real oriented blow-up. Putting \par
\noindent
$D^{r.d.}=\pi^{-1}\left((e^{\frac{\pi}{2}\ii},e^{\frac{3\pi}{2}\ii})\infty\right)\subset \pi^{-1}(D^\prime)$, we define the p-th rapid decay homology group by the formula
\begin{equation}
\Homo_{p}^{r.d.}\left(\T^n,(\mathcal{O}_{(\Gm)^n},\nabla_z^\vee)\right)=\Homo_{p}\left(\T^n\cup D^{r.d.},D^{r.d.};\C x^c\right).
\end{equation}
The isomorphism (\ref{RDduality}) is given in terms of integration. Note that $\Homo^p\left(\Omega^{\bullet}((\Gm)^n),\nabla_z\right)$ is also denoted by
$\Homo^p_{dR}\left((\Gm)^n,(\mathcal{O}_{(\Gm)^n},\nabla_z)\right).$ Combining \cref{theorem:fundamental} (2), (\ref{isom1}), (\ref{deRham}), (\ref{RDduality}), we have

\begin{prop}\label{prop:solution}
For any $z\in\Omega^{an},$ we have
\begin{equation}
\mathcal{H}om_{\mathcal{D}_{\C^N}}(\left(M_A(c)\right)^{an},\mathcal{O}_{\C^N})_z\simeq\Homo_{n}^{r.d.}\left(\T^n,(\mathcal{O}_{(\Gm)^n},\nabla_z^\vee)\right).
\end{equation}
\end{prop}
\noindent
In general situation, one must consider toric blow-up to construct a good compactification of $(\Gm)^n$ to compute the rapid decay homology group. However, when Newton polytope has a very simple form, we can compute the space of rapid decay cycles as we shall see below.

\begin{lem}\label{lem:Kunneth}
\indent
Let $U_1,U_2$ be smooth Affine algebraic varieties over $\mathbb{C}$ and consider flat algebraic connections $(E_i,\nabla_i)$ on $U_i$ $(i=1,2)$.
There is a canonical isomorphism 
\begin{equation}
\displaystyle\bigoplus_{k_1+k_2=m}\Homo^{r.d.}_{k_1}\left(U_1,(E_1^\vee,\nabla_1^\vee)\right)\boxtimes\Homo^{r.d.}_{k_2}\left(U_2,(E_2^\vee,\nabla_2^\vee)\right)\overset{\varphi}{\tilde{\rightarrow}}\Homo^{r.d}_{m}\left(U_1\times U_2,((E_1\boxtimes E_2)^\vee,(\nabla_1\boxtimes\nabla_2)^\vee)\right)
\end{equation}
which is characterized by the formula
\begin{equation}
\int_{\varphi(\Gamma_1\otimes\Gamma_2)}\omega_1\wedge\omega_2=\biggl(\int_{\Gamma_1}\omega_1\biggr)\biggl(\int_{\Gamma_2}\omega_2\biggr),
\end{equation}
where $\Gamma_i\in\Homo^{r.d.}_{k_i}\left(U_i,(E_i^\vee,\nabla_i^\vee)\right)$ and $\omega_i\in\Homo^{k_i}_{dR}\left(U_i,(E_i,\nabla_i)\right)$ (i=1,2). $\varphi(\Gamma_1\otimes\Gamma_2)$ will also be denoted by $\Gamma_1\times\Gamma_2.$
\end{lem}

\begin{proof}
We can easily see that the wedge product gives an isomorphism of algebraic de Rham complexes 
\begin{equation}
(\Gamma(U_1,\Omega^\bullet(E_1)),\nabla_1)\boxtimes(\Gamma(U_2,\Omega^\bullet(E_2)),\nabla_2)\overset{\sim}{\rightarrow}(\Gamma(U_1\times U_2,\Omega^\bullet(E_1\boxtimes E_2)),\nabla_1\boxtimes\nabla_2).
\end{equation}
Taking their $m$-th cohomology and dual, we obtain the lemma. 
\end{proof}

Remember now that the Newton non-degenerate locus is decomposed into a direct product as we explain below. We put $\tau\subset\{ 1,\cdots, N\}$ to be the set $\tau=\biggl\{i|{\bf a}(i)\notin\displaystyle \bigcup_{\substack{\Gamma <\Delta_A\\ 0\notin\Gamma}}\Gamma\biggr\}.$ We regard $\Omega^{\bar{\tau}}\overset{\rm def}{=}\Omega\cap\displaystyle\bigcap_{i\in\tau}\{z_i=0\}$ as a  subset of $\mathbb{A}^{\bar{\tau}}$. Then, by the definition of non-degeneracy, we have
\begin{equation}
\Omega=\mathbb{A}^{\tau}\times\Omega^{\bar{\tau}}.
\end{equation}
Here, $\Omega^{\bar{\tau}}=\left\{z_{\bar{\tau}}\in\A^{\bar{\tau}}|h^{\bar{\tau}}_{z}(x)=\sum_{j\in\bar{\tau}}z_jx^{{\bf a}(j)}\text{ is Newton non-degenerate}\right\}.$

\noindent
Now we introduce the following assumption $(*)$:
\begin{quote}
$(*)$Laurent polynomial $h_z(x)$ is given by 
\begin{equation}
h_z(x)=\displaystyle\sum_{i=1}^l(z_i^+x_i^{a_i^+}+z_i^-x_i^{-a_i^-})+\sum_{i=l+1}^nz_ix_i^{a_i}+\sum_{j=n+1}^{N-l}z_jx^{{\bf a}(j)}, 
\end{equation}
where $a_i^{\pm},\; a_i\in\Z_{>0}$ and ${\bf a}(j)\in\Z^{n\times 1}.$ Moreover, if one puts 
\begin{equation}
\Delta=c. h.\left\{0, a_1^+{\bf e}(1),-a_1^-{\bf e}(1),\dots,a_l^+{\bf e}(l),-a_l^-{\bf e}(l),a_{l+1}{\bf e}(l+1),\dots,a_n{\bf e}(n)\right\},
\end{equation}
then one must have ${\bf a}(j)\in\Int\Delta\cup\displaystyle \bigcup_{\substack{\Gamma <\Delta\\ 0\in\Gamma}}\Gamma.$ Here, ${\bf e}(i)$ is the standard basis of $\Z^{n\times 1}.$ 

\end{quote}
Under this assumption, it can readily be seen that a parameter vector $c$ is non-resonant if $c\notin\displaystyle\bigcup_{i=l+1}^n\Bigl(\Z^{n\times 1}+\C{\bf e}(1)+\cdots+\C\widehat{{\bf e}(i)}+\cdots+\C{\bf e}(n)\Bigr)$ and $\Omega$ has the form $\Omega=\A^{N-n-l}\times\Omega^\prime$ where $\Omega^\prime$ is a subset of $\A^{n+l}$ with coordinates $(z_1^+,z_1^-,\dots,z_l^+,z_l^-,z_{l+1},\dots,z_n)$ defined by 
\begin{equation}
\Omega^\prime=\left\{(z_1^+,z_1^-,\dots,z_l^+,z_l^-,z_{l+1},\dots,z_n)\in\A^{n+l}\mid z_1^+z_1^-\cdots z_l^+z_l^-z_{l+1}\cdots z_n\neq0\right\}.
\end{equation}

Now let us take a point $\overset{\circ}{z}=(0,\overset{\circ}{z}_{\bar{\tau}})\in\Omega.$ By \cref{prop:solution}, we know that 

\noindent
$\Homo^{r.d.}_n(\mathbb{T}^n, (\mathcal{O}_{(\Gm)^n}\nabla_{\overset{\tiny \circ}{z}}^\vee))$ corresponds to solutions of $M_A(c).$ For any given data $\gamma\in\C, m_1,m_2,m\in\Z_{>0},w_1,w_2,w\in\C^*,$ we define connections on $\mathcal{O}_{\Gm}$ by
\begin{equation}
\nabla(\gamma;(w_1,m_1),(w_2,m_2))=d+\gamma \frac{dy}{y}\wedge+m_1w_1y^{m_1-1}dy\wedge-m_2w_2y^{-m_2-1}dy\wedge
\end{equation}
and by
\begin{equation}
\nabla(\gamma;(w,m))=d+\gamma \frac{dy}{y}\wedge+mwy^{m-1}dy\wedge,
\end{equation}
where $y$ is a coordinate of $\Gm.$ By definition, we have 
\begin{equation}
\Homo^0_{dR}\Bigl(\Gm,\bigl(\mathcal{O}_{\Gm},\nabla(\gamma;(w_1,m_1),(w_2,m_2))\bigr)\Bigr)=0
\end{equation}
and
\begin{equation}
\Homo^0_{dR}\Bigl(\Gm,(\mathcal{O}_{\Gm},\nabla(\gamma;(w,m)))\Bigr)=0
\end{equation}
for any $\gamma\in\C, m_1,m_2,m\in\Z_{>0},w_1,w_2,w\in\C^*.$ On the other hand, we can explicitly describe the 1st homology groups
\begin{equation}
\Homo^{r.d.}_{1}\Bigl(\T,(\mathcal{O}_{\Gm},\nabla(\gamma;(w_1,m_1),(w_2,m_2))^\vee)\Bigr)
\end{equation}
and
\begin{equation}
\Homo^{r.d.}_{1}\Bigl(\T,(\mathcal{O}_{\Gm},\nabla(\gamma;(w,m))^\vee)\Bigr).
\end{equation}

\noindent

For any positive integers $k,m$, we put $\gamma_{k,m}=\exp\{(\frac{2k-1}{m})\pi\sqrt{-1}\}$, and 

\noindent
$\Gamma_{k,m}=-[\varepsilon, \infty)e^{\varepsilon^\prime\sqrt{-1}}\gamma_{k,m}+C_{k,m}+[\varepsilon,\infty)e^{-\varepsilon^\prime\sqrt{-1}}\gamma_{k+1,m}$ where $\varepsilon$ and $\varepsilon^\prime$ are small positive numbers and
\begin{equation}
C_{k,m}(t)=\varepsilon e^{(\frac{2(k+t)-1}{m})\pi\sqrt{-1}}\; (\frac{m\varepsilon^\prime}{2}\leq t\leq 1-\frac{m\varepsilon^\prime}{2})
\end{equation}
is an arc. Note that $\Gamma_{k,m}$ is regarded as a relative chain. Then, if $w>0,$ we have
\begin{equation}
\Homo^{r.d.}_{1}\Bigl(\T,(\mathcal{O}_{\Gm},\nabla(\gamma;(w,m))^\vee)\Bigr)=\bigoplus_{k=0}^{m-1}\C\Gamma_{k,m}\otimes y^\gamma.
\end{equation}
Moreover, if $\gamma\notin\Z$ and $w_1,w_2>0,$ then we have
\begin{equation}
\Homo^{r.d.}_{1}\Bigl(\T,(\mathcal{O}_{\Gm},\nabla(\gamma;(w_1,m_1),(w_2,m_2))^\vee)\Bigr)=\bigoplus_{k=0}^{m_1-1}\C\Gamma^+_{k,m_1}\otimes y^\gamma\oplus\bigoplus_{l=0}^{m_2-1}\C\Gamma^-_{l,m_2}\otimes y^\gamma.
\end{equation}

\noindent
Here, by abuse of notation we put $\Gamma^+_{k,m}=\Gamma_{k,m}$ and $\Gamma^-_{k,m}$ stands for the image of the cycle $\Gamma_{k,m}$ through the isomorphism $y\mapsto \frac{1}{y}.$ Now by \cref{lem:Kunneth}, we see that if $z_{\bar{\tau}}>0,$ $\Homo^{r.d.}_n(\mathbb{T}^n, (\mathcal{O}_{(\Gm)^n}\nabla_{\overset{\tiny \circ}{z}}^\vee))$ is decomposed as

\begin{eqnarray}
\hspace{-5mm}& &\hspace{-3mm}\Homo^{r.d.}_n\Bigl(\mathbb{T}^n, (\mathcal{O}_{(\Gm)^n}\nabla_{\overset{\tiny \circ}{z}}^\vee)\Bigr)\nonumber\\
\hspace{-5mm}&=&\hspace{-3mm}\Homo^{r.d.}_1\Bigl(\T, (\mathcal{O}_{\Gm},\nabla(c_1;(\overset{\tiny \circ}{z}_1^+,a_1^+),(\overset{\tiny \circ}{z}_1^-,a_1^-))^\vee)\Bigr)\boxtimes\dots\boxtimes\Homo^{r.d.}_1\Bigl(\T, (\mathcal{O}_{\Gm},\nabla(c_l;(\overset{\tiny \circ}{z}_l^+,a_l^+),(\overset{\tiny \circ}{z}_l^-,a_l^-))^\vee)\Bigr)\nonumber\\
\hspace{-5mm} & &\hspace{-3mm}\boxtimes\Homo^{r.d.}_1\Bigl(\T, (\mathcal{O}_{\Gm},\nabla(c_{l+1};(\overset{\tiny \circ}{z}_{l+1},a_{l+1}))^\vee)\Bigr)\boxtimes\dots\boxtimes\Homo^{r.d.}_1\Bigl(\T, (\mathcal{O}_{\Gm},\nabla(c_n;(\overset{\tiny \circ}{z}_n,a_n))^\vee)\Bigr)\label{decomposition}.\hspace{-3mm}
\end{eqnarray}
Therefore, if $c_i\notin\Z\; (i=1,\dots,l)$, we have the following

\begin{thm}\label{bthm:cycles}
Under the assumption $(*)$ and the assumption that $c_i\notin\Z\; (i=1,\dots,l),$ if $z_{\bar{\tau}}>0,$ we have

\begin{equation}
\Homo^{r.d.}_n\Bigl(\mathbb{T}^n, (\mathcal{O}_{(\Gm)^n}\nabla_{\overset{\tiny \circ}{z}}^\vee)\Bigr)=\bigoplus_{
\substack{
\varepsilon_i=\pm  \\
k_i^{\varepsilon_i}=0,\dots,a_i^{\varepsilon_i}-1 \; (i=1,\dots,l)\\
k_j^{\varepsilon_i}=0,\dots,a_j-1 \; (j=l+1,\dots,n)
}
}
\C\displaystyle\left(\prod_{i=1}^l\Gamma_{k_i^{\varepsilon_i},a_i^{\varepsilon_i}}^{\varepsilon_i}\right)\times\left(\prod_{j=l+1}^n\Gamma_{k_j,a_j}\right).
\end{equation}

\end{thm}

\noindent
Now we consider the special example treated by D. Kaminski and R. B. Paris. \cite{KP}

\begin{exa}\label{diamond}
{\rm
$n=2,$ $N=4,$ $c=^{t}(c_1,c_2)\in\mathbb{C}^2,$ $c_i\notin\mathbb{Z},$ $m_1,m_2\in\mathbb{Z}_{>0}$ such that $\frac{2}{m_1}+\frac{1}{m_2}<1,$ 
$
A=
\begin{pmatrix}
m_1 & 0 &1 &2\\
0 & m_2 &1 &1
\end{pmatrix}.
$
The constraint on $m_1$ and $m_2$ means that the lattice points ${}^t(1,1)$ and ${}^t(2,1)$ are inside the Newton polytope $\Delta=c.h.\{0, {}^t(m_1,0), {}^t(0,m_2)\}$. Now, the general solution is given by
\begin{equation}
f(z)=\int_\Gamma\exp\{ z_1x^{m_1}+z_2y^{m_2}+z_3xy+z_4x^2y\}x^{c_1-1}y^{c_2-1}dxdy.
\end{equation} 

Moreover, by \cref{bthm:cycles}, if $z_1,z_2>0,$ we know that $\Bigl\{\Gamma_{k,m_1}\times\Gamma_{l,m_2}\Bigr\}_{\substack{k=0,1,\cdots,m_1-1\\
l=0,1,\cdots,m_2-1}}$ forms a basis of rapid decay homology group.

The readers should be careful about the choice of the argument. Let us consider the integral over $\Gamma_{0,m_1}\times\Gamma_{0,m_2}.$ By the change of variables $\xi=x^{m_1},\;\eta=y^{m_2},$ and the formula

\begin{equation}
e^z=\frac{1}{2\pi\sqrt{-1}}\int_{\Gamma_{0,1}}\Gamma(s)(e^{\pi\sqrt{-1}}z)^{-s}ds,
\end{equation}
we have

\begin{align}
&f_{0,0}(z)\nonumber\\
\overset{\rm def}{=}&\int_{\Gamma_{0,m_1}\times\Gamma_{0,m_2}}\exp\{ z_1x^{m_1}+z_2y^{m_2}+z_3xy+z_4x^2y\}x^{c_1-1}y^{c_2-1}dxdy\\
 =&\frac{1}{m_1m_2}\int_{\Gamma_{0,1}\times\Gamma_{0,1}}\exp\Bigl\{ z_1\xi+z_2\eta+z_3\xi^{\frac{1}{m_1}}\eta^\frac{1}{m_2}+z_4\xi^\frac{2}{m_1}\eta^\frac{1}{m_2}\Bigr\}\xi^{\frac{c_1}{m_1}-1}\eta^{\frac{c_2}{m_2}-1}d\xi d\eta\\
 =&\frac{1}{(2\pi\sqrt{-1})^2 m_1m_2}\int_{\Gamma_{0,1}\times\Gamma_{0,1}}e^{ z_1\xi+z_2\eta}\nonumber\\
  &\Bigl\{ \int_{\Gamma_{0,m_1}\times\Gamma_{0,m_2}}\Gamma(s)\Gamma(t)\Bigl(e^{\pi\sqrt{-1}}z_3\xi^{\frac{1}{m_1}}\eta^\frac{1}{m_2}\Bigr)^{-s}\Bigl(e^{\pi\sqrt{-1}}z_4\xi^\frac{2}{m_1}\eta^\frac{1}{m_2}\Bigr)^{-t}dsdt\Bigr\}\xi^{\frac{c_1}{m_1}-1}\eta^{\frac{c_2}{m_2}-1}d\xi d\eta\\
 =&\frac{1}{(2\pi\sqrt{-1})^2 m_1m_2}\int_{\Gamma_{0,1}\times\Gamma_{0,1}}dsdt\Gamma(s)\Gamma(t)\Bigl(e^{\pi\sqrt{-1}}z_3\Bigr)^{-s}\Bigl(e^{\pi\sqrt{-1}}z_4\Bigr)^{-t}\nonumber\\
  &\int_{\Gamma_{0,1}\times\Gamma_{0,1}}e^{ z_1\xi+z_2\eta}\xi^{\frac{c_1}{m_1}-\frac{s}{m_1}-\frac{2t}{m_1}-1}\eta^{\frac{c_2}{m_2}-\frac{s}{m_2}-\frac{t}{m_2}-1}d\xi d\eta\\
 =&\frac{z_1^{-\frac{c_1}{m_1}}z_2^{-\frac{c_2}{m_2}}}{m_1m_2}\int_{\Gamma_{0,1}\times\Gamma_{0,1}}\frac{\Gamma(s)\Gamma(t)}{\Gamma(1-\frac{c_1}{m_1}+\frac{s}{m_1}+\frac{2t}{m_1})\Gamma(1-\frac{c_2}{m_2}+\frac{s}{m_2}+\frac{t}{m_2})}\nonumber\\
  &\Bigl( e^{\pi\sqrt{-1}}z_1^{-\frac{1}{m_1}}z_2^{-\frac{1}{m_2}}z_3\Bigr)^{-s}\Bigl( e^{\pi\sqrt{-1}}z_1^{-\frac{2}{m_1}}z_2^{-\frac{1}{m_2}}z_4\Bigr)^{-t}dsdt.
\end{align}

\noindent
Here, we have used the following simple

\begin{lem}\label{lemma}
Suppose $\alpha\in\mathbb{C}.$ Then we have

\begin{equation}
\int_{\Gamma_{0,1}}\xi^{\alpha-1}e^\xi d\xi=\frac{2\pi\sqrt{-1}}{\Gamma(1-\alpha)}.
\end{equation}
\end{lem}

\noindent
Note that in the lemma above, the argument of $\xi$ is from $-\pi$ to $\pi$. If we take another integration contour, say $\Gamma_{k,1}$, we have
\begin{equation}
\int_{\Gamma_{k,1}}\xi^{\alpha-1}e^{z\xi} d\xi=(e^{-2k\pi\sqrt{-1}}z)^{-\alpha}\frac{2\pi\sqrt{-1}}{\Gamma(1-\alpha)}
\end{equation}
for $z>0$. Thus, we can conclude that
\begin{eqnarray*}
f_{k,l}(z)&\overset{\rm def}{=}&\int_{\Gamma_{k,m_1}\times\Gamma_{l,m_2}}\exp\{ z_1x^{m_1}+z_2y^{m_2}+z_3xy+z_4x^2y\}x^{c_1-1}y^{c_2-1}dxdy\\
 &=&f_{0,0}(e^{-2k\pi\sqrt{-1}}z_1,e^{-2l\pi\sqrt{-1}}z_2,z_3,z_4).
\end{eqnarray*}
We can confirm that Mellin-Barnes integral representation of $f$ is convergent for $z_1z_2z_3z_4\neq 0$ as we shall see in Section \ref{convergence}.  From the last formula, we can easily give its series representation
\begin{equation}
f_{0,0}=\frac{(2\pi\sqrt{-1})^2}{m_1m_2}z_1^{-\frac{c_1}{m_1}}z_2^{-\frac{c_2}{m_2}}\displaystyle\sum_{i,j=0}^\infty \frac{\Bigl( z_1^{-\frac{1}{m_1}}z_2^{-\frac{1}{m_2}}z_3\Bigr)^{i}\Bigl( z_1^{-\frac{2}{m_1}}z_2^{-\frac{1}{m_2}}z_4\Bigr)^{j}}{\Gamma(1-\frac{c_1}{m_1}-\frac{i}{m_1}-\frac{2j}{m_1})\Gamma(1-\frac{c_2}{m_2}-\frac{i}{m_2}-\frac{j}{m_2})i!j!}.
\end{equation}

%Remember now that Newton non-degeneracy is a condition only on the faces of $\Delta$. In our case, $\Omega=(\mathbb{C}^*)^2\times\mathbb{C}^2$. So if we choose a point $(\overset{\circ}{z}_1, \overset{\circ}{z}_2, 0,0)\in\Omega,$ by Cauchy-Kowalevsky-Kashiwara theorem and projection formula, we have

%By substituting formula (\ref{mellin}), we have
%$
%\begin{aligned}
%\mathcal{H}om_\mathcal{D}()
%\end{aligned}
%$

}
\end{exa}

\section{Deduction of Mellin-Barnes integral representations}

\indent
In this section, we give a psychological deduction of Mellin-Barnes integral representations. Though it will turn out that Mellin-Barnes integral can be mathematically justified, the convergence of the corresponding Laplace type integral is a more subtle problem. Let us take any $\sigma\subset\{ 1,\cdots ,N \}$ such that $|\sigma|=n$, $\det A_\sigma\neq 0$,  and $\sigma\subset\displaystyle\bigcup_{\substack{\Gamma<\Delta_A\\ 0\notin\Gamma}}\Gamma.$ We consider a covering transform $\xi_i=x^{{\bf a}(i)}\; (i\in\sigma).$ We abbreviate this covering transform as $\xi=x^{A_\sigma}$. We have
\begin{equation}
d\xi_j=\displaystyle\sum_{i=1}^na_{ij}x^{{\bf a}(j)-{\bf e}(i)}dx_i.
\end{equation}
Therefore, 
\begin{equation}
d\xi_\sigma=\displaystyle\bigwedge_{i\in\sigma}d\xi_i=\det (A_\sigma)x^{A_\sigma{\bf 1}-{\bf 1}}dx
\end{equation}
which implies 
\begin{equation}
dx=\frac{1}{\det (A_\sigma)}\xi_\sigma^{A_\sigma^{-1}{\bf 1}-{\bf 1}}d\xi_\sigma.
\end{equation}

Let us formally construct integration cycles. Consider a cycle $\Gamma_0=\overset{n \text{ times}}{\overbrace{\Gamma_{0,1}\times\cdots\times\Gamma_{0,1}}}$ in $\xi$ space, on which all the branches of $\xi_\sigma^{A_\sigma^{-1}{\bf a}(i)}$ takes its principal value. Now consider a covering map 
\begin{equation}
p:\; \mathbb{T}^n_x\ni x\mapsto\xi_\sigma=x^{A_\sigma}\in\mathbb{T}^n,
\end{equation}
Suppose that we can take a rapid decay cycle $\Gamma\in\Homo_{n}^{r.d.}\left(\T^n,(\mathcal{O}_{(\Gm)^n},\nabla_z^\vee)\right)$ such that $p_*\Gamma=\Gamma_0$. On this contour, the dominant behaviour of $h_z(x)$ is given by $\displaystyle\sum_{i\in\sigma}z_ix^{{\bf a}(i)}$. By the construction of $\Gamma$, we formally have

\begin{eqnarray}
f_0(z)&\overset{\rm def}{=}&\frac{1}{(2\pi\sqrt{-1})^n}\int_\Gamma e^{h_z(x)}x^{\bf c-1}dx\\
 &=&\frac{1}{(2\pi\sqrt{-1})^n}\int_{\Gamma_0}e^{h_z(\xi_\sigma^{A_\sigma^{-1}})}\xi_\sigma^{A_\sigma^{-1}{\bf c}-{\bf 1}}\frac{d\xi_\sigma}{\det A_\sigma}\label{Laplace}\\
 &=&\frac{1}{(2\pi\sqrt{-1})^N}\int_{\overset{N-n \text{ times}}{\overbrace{\Gamma_{0,1}\times\cdots\times\Gamma_{0,1}}}}\Gamma(s_{\bar{\sigma}})\biggl(e^{\pi\sqrt{-1}{\bf 1}_{\bar{\sigma}}}z_{\bar{\sigma}}\biggr)^{-s_{\bar{\sigma}}}\times\nonumber\\
& &\int_{\overset{n\text{ times}}{\overbrace{\Gamma_{0,1}\times\cdots\times\Gamma_{0,1}}}}\xi_\sigma^{A_\sigma^{-1}(c-A_{\bar{\sigma}}s_{\bar{\sigma}})-{\bf 1}}\exp\{\displaystyle\sum_{i\in\sigma}z_i\xi_i\}\frac{d\xi_\sigma}{\det A_\sigma}ds_{\bar{\sigma}}\\
 &=&\frac{z_\sigma^{-A_\sigma^{-1}c}}{(2\pi\sqrt{-1})^{N-n}}\int_{\overset{N-n \text{ times}}{\overbrace{\Gamma_{0,1}\times\cdots\times\Gamma_{0,1}}}}\frac{\Gamma(s_{\overline{\sigma}})(e^{\pi\sqrt{-1}{\bf 1}_{\bar{\sigma}}}z_{\overline{\sigma}})^{-s_{\overline{\sigma}}}z_\sigma^{A_\sigma^{-1}A_{\overline{\sigma}}s_{\overline{\sigma}}}}{\Gamma({\bf 1}-A_\sigma^{-1}(c-A_{\overline{\sigma}}s_{\overline{\sigma}}))}\frac{ds_{\overline{\sigma}}}{\det A_\sigma}.\label{MBint}
\end{eqnarray}

\noindent
We used the Lemma \ref{lemma} above. We call (\ref{MBint}) Mellin-Barnes integral representation.

\begin{rem}
We will see that (\ref{MBint}) converges in the following sections, whereas (\ref{Laplace}) does not need to converge. As a simple example, let us put $n=2$, $N=3$, generic $c$ and 
$
A=
\begin{pmatrix}
1& 0 &-1\\
0& 1 &2
\end{pmatrix}.
$
We take $\sigma=\{1,2\}.$ Then, we have $\xi_1=x_1$ and $\xi_2=x_2.$ Our integration contour is nothing but $\Gamma_0=\Gamma_{0,1}\times\Gamma_{0,1}.$ However, we can easily see that (\ref{Laplace}) in this case can never be convergent for any $(z_1,z_2,z_3)\in(\C^*)^3$ since $e^{z_3x_1^{-1}x_2^2}$ is not bounded on $\Gamma_0.$

\end{rem}

\section{Mellin-Barnes integral satisfies GKZ}\label{equation}
Take any $\sigma\subset\{ 1,\cdots ,N \}$ such that $|\sigma|=n$, $\det A_\sigma\neq 0$, and $\sigma\subset\displaystyle\bigcup_{\substack{\Gamma<\Delta_A\\ 0\notin\Gamma}}\Gamma$. The last condition is equivalent to the condition that for any $j\in\bs$, $|A_\s^{-1}{\bf a}(j)|\leq 1$.

\noindent
We put
\begin{equation}
F_{\s,0}(z)=\frac{z_\sigma^{-A_\sigma^{-1}c}}{(2\pi\sqrt{-1})^{N-n}}\int_{C_{\overline{\sigma}}}\frac{\Gamma(s_{\overline{\sigma}})(e^{\pi\sqrt{-1}}z_{\overline{\sigma}})^{-s_{\overline{\sigma}}}z_\sigma^{A_\sigma^{-1}A_{\overline{\sigma}}s_{\overline{\sigma}}}}{\Gamma({\bf 1}-A_\sigma^{-1}(c-A_{\overline{\sigma}}s_{\overline{\sigma}}))}ds_{\overline{\sigma}}.
\end{equation}
\noindent
Here, $C_{\overline{\sigma}}$ is a product of 1 dimensional integration contours $\overset{N-n \text{ times}}{\overbrace{\Gamma_{0,1}\times\cdots\times\Gamma_{0,1}}}$. For any $l\in\Z_{>0}$ vector ${\bf v}\in\mathbb{C}^l$, we put $|{\bf v}|=\sum_{j=1}^lv_j.$ We   put $s_j=|A_\sigma^{-1}{\bf a}(j)|\leq 1(j\notin\sigma).$ In this section, we show that this function satisfies $M_A(c)$ assuming that the integral is convergent. The first equation is very easy to check. Indeed, for any $t=(t_1,\cdots ,t_n)\in(\mathbb{C}^{*})^n$, one has
\begin{align}
&F_{\s,0}(t^{{\bf a}(1)}z_1,\cdots ,t^{{\bf a}(N)}z_N)\nonumber\\
=&\frac{t^{-c}z_\sigma^{-A_\sigma^{-1}c}}{(2\pi\sqrt{-1})^{N-n}}\int_{C_{\overline{\sigma}}}\frac{\Gamma(s_{\overline{\sigma}})(e^{\pi\sqrt{-1}}z_{\overline{\sigma}})^{-s_{\overline{\sigma}}}t^{-A_{\barsigma} s_{\barsigma}}t^{A_\sigma A_\sigma^{-1}A_{\barsigma}s_{\barsigma}}z_\sigma^{A_\sigma^{-1}A_{\overline{\sigma}}s_{\overline{\sigma}}}}{\Gamma({\bf 1}-A_\sigma^{-1}(c-A_{\overline{\sigma}}s_{\overline{\sigma}}))}ds_{\overline{\sigma}}\\
=&t^{-c}F(z_1,\cdots ,z_N).
\end{align}

Therefore, by applying $\frac{\partial}{\partial t_i}|_{t_i=1}$, we obtain $E_iF(z)=0.$

Next we check the second equation. Take any $u=u_+-u_-\in L_A$ and decompose it as $u=u^\sigma+u^{\barsigma}$ and we decompose it further as $u^{\sigma}=u^{\sigma}_+-u^\sigma_-$, $u^{\barsigma}=u^{\barsigma}_+-u^{\barsigma}_-.$ Note that $u^\sigma=-A_\sigma^{-1}A_{\barsigma}u^{\barsigma}$ implies $A_\sigma^{-1}A_{\barsigma}u^{\barsigma}_++u^\sigma_+=A_\sigma^{-1}A_{\barsigma}u^{\barsigma}_-+u^\sigma_-$.

We define some notations as follows:
$[a]_u=\prod_{j=1}^u(a-j+1)=\frac{\Gamma(a+1)}{\Gamma(a-u+1)} (a\in\mathbb{C}, u\in\mathbb{Z}_{\geq 1}),$ $[a]_0=1$, $[{\bf a}]_{{\bf u}}=\prod_{i=1}^l[a_i]_{u_i} ({\bf a}\in\mathbb{C}^l, u\in\mathbb{Z}_{\geq 1}^l).$

Then, 
\begin{align}
 &(2\pi\sqrt{-1})^{N-n}\partial^{u_+}F_{\s,0}(z)\nonumber\\
= &\int_{C_{\overline{\sigma}}}\frac{\Gamma(s_{\overline{\sigma}})[A^{-1}_\sigma(A_{\barsigma}s_{\barsigma}-c)]_{u^\sigma_+}[s_{\barsigma}+u^{\barsigma}_+-{\bf 1}]_{u^{\barsigma}_+}}{\Gamma({\bf 1}-A_\sigma^{-1}(c-A_{\overline{\sigma}}s_{\overline{\sigma}}))}(e^{\pi\sqrt{-1}{\bf 1}_{\barsigma}}z_{\overline{\sigma}})^{-s_{\overline{\sigma}}-u^{\barsigma}_+}z_\sigma^{A_\sigma^{-1}(A_{\overline{\sigma}}s_{\overline{\sigma}}-c)-u^\sigma_+}ds_{\overline{\sigma}}\\
=&\int_{C_{\overline{\sigma}}+u^{\barsigma}_+}\frac{\Gamma(t_{\overline{\sigma}}-u^{\barsigma}_+)[A^{-1}_\sigma(A_{\barsigma}(t_{\barsigma}-u_+^{\barsigma})-c)]_{u^\sigma_+}[t_{\barsigma}-{\bf 1}]_{u^{\barsigma}_+}}{\Gamma({\bf 1}-A_\sigma^{-1}(c-A_{\overline{\sigma}}s_{\overline{\sigma}}))}(e^{\pi\sqrt{-1}{\bf 1}_{\barsigma}}z_{\overline{\sigma}})^{-t_{\overline{\sigma}}}\nonumber\\
  & \times z_\sigma^{A_\sigma^{-1}(A_{\overline{\sigma}}t_{\overline{\sigma}}-c)-A^{-1}_\sigma A_{\barsigma}u^{\barsigma}_+-u^\sigma_+}dt_{\overline{\sigma}}\\
=&\int_{C_{\overline{\sigma}}+u^{\barsigma}_+}\frac{\Gamma(t_{\overline{\sigma}}-u^{\barsigma}_+)\Gamma({\bf 1}+A^{-1}_\sigma(A_{\barsigma}(t_{\barsigma}-u_+^{\barsigma})-c))\Gamma(t_{\barsigma})}{\Gamma({\bf 1}-A_\sigma^{-1}(c-A_{\overline{\sigma}}s_{\overline{\sigma}}))\Gamma({\bf 1}+A^{-1}_\sigma(A_{\barsigma}(t_{\barsigma}-u_+^{\barsigma})-c)-u_+^{\sigma})\Gamma(t_{\barsigma}-u_+^{\barsigma})}(e^{\pi\sqrt{-1}{\bf 1}_{\barsigma}}z_{\overline{\sigma}})^{-t_{\overline{\sigma}}}\nonumber\\
  &\times z_\sigma^{A_\sigma^{-1}(A_{\overline{\sigma}}t_{\overline{\sigma}}-c)-A^{-1}_\sigma A_{\barsigma}u^{\barsigma}_+-u^\sigma_+}dt_{\overline{\sigma}}\\
=&\int_{C_{\overline{\sigma}}+u^{\barsigma}_+}\frac{\Gamma(t_{\barsigma})}{\Gamma({\bf 1}+A^{-1}_\sigma(A_{\barsigma}t_{\barsigma}-c)-A_{\barsigma}u^{\barsigma}_+-u_+^{\sigma})}(e^{\pi\sqrt{-1}{\bf 1}_{\barsigma}}z_{\overline{\sigma}})^{-t_{\overline{\sigma}}}z_\sigma^{A_\sigma^{-1}(A_{\overline{\sigma}}t_{\overline{\sigma}}-c)-A^{-1}_\sigma A_{\barsigma}u^{\barsigma}_+-u^\sigma_+}dt_{\overline{\sigma}}\\
=&(2\pi\sqrt{-1})^{N-n}\partial^{u_-}F(z)
\end{align}

where we put $t_{\barsigma}=s_{\barsigma}+u_+^{\barsigma}.$ The last equality follows from the homotopy of integration domain. 

It is important to notice that Mellin-Barnes integral representation tells us that the solution is holomorphic (meromorphic) with respect to parameters $c.$ Thus, the integral remains to be the solution of GKZ system as long as the parameters $c$ lie in ``non-singular locus''.

\begin{defn}
We say $c$ is very generic when $A_\sigma^{-1}(c+A_{\bar{\sigma}}{\bf m})$ does not contain any integer entry for any ${\bf m}\in\mathbb{Z}_{\geq 0}^{\bar{\sigma}}.$
\end{defn}

In the followings, we always assume that $c$ is very generic. This ensures that the resulting $\Gamma$ series are all non-zero.

\section{Mellin-Barnes integral is convergent}\label{convergence}
\indent
First, we introduce the following key formula for multi-Beta following \cite{B}.

\begin{lem}\label{blem:multibeta}
Let $\Delta^k\subset\mathbb{R}^k$ be the k simplex $\Delta^k=\{(x_1,\cdots ,x_k)\in\mathbb{R}^k|x_1,\cdots,x_k\geq 0, \sum_{j=1}^kx_j\leq 1\}$ and $\alpha_1,\cdots ,\alpha_{k+1}\in\mathbb{C}$. If we denote by $P_k$ the Pochhammer cycle associated to $\Delta^k$ in the sense of \cite{B}, we have
\begin{align}
\frac{\Gamma(\alpha_1)\cdots\Gamma(\alpha_{k+1})}{\Gamma(\alpha_1+\cdots +\alpha_{k+1})}&=\int_{\Delta^k}t_1^{\alpha_1-1}\cdots t_{k}^{\alpha_k-1}(1-t_1\cdots -t_k)^{\alpha_{k+1}-1}dt_1\cdots dt_k\\
&=\prod_{j=1}^{k+1}\frac{1}{(1-e^{-2\pi\sqrt{-1}\alpha_j})}\int_{P_k}t_1^{\alpha_1-1}\cdots t_{k}^{\alpha_k-1}(1-t_1\cdots -t_k)^{\alpha_{k+1}-1}dt_1\cdots dt_k.
\end{align}
The first equality is valid if $\Re(\alpha_j)>0$. The second equality is valid if $\alpha\notin\Z$.
\end{lem}

Let $(A_\sigma^{-1}c)_l$ denote the $l$-th entry of the vector  $A_\sigma^{-1}c$. Since $c$ is weakly very generic, we have $A_\sigma^{-1}c\notin\mathbb{Z}.$ Let me put $(A_\sigma^{-1}A_{\barsigma}s_{\barsigma})_l=\sum_{i\notin\sigma}\alpha_{li}s_i.$ Then, we have 
\begin{align}
&\frac{\Gamma(s_{\barsigma})}{\Gamma({\bf 1}-A_\sigma^{-1}(c-A_{\barsigma}s_{\barsigma}))}\nonumber\\
=&\prod_{l=1}^n\frac{(\prod_{i\notin\sigma}\Gamma(\alpha_{li}s_i))\Gamma(1-(A_\sigma^{-1}c)_l)}{\Gamma(1-(A_\sigma^{-1}c)_l+\sum_{i\notin\sigma}\alpha_{li}s_i)}\prod_{i\notin\sigma}\frac{\Gamma(s_i)}{\prod_{l=1}^n\Gamma(\alpha_{li}s_i)}\prod_{l=1}^n\frac{1}{\Gamma(1-(A_\sigma^{-1}c)_l)}.
\end{align}

\noindent
By the \cref{blem:multibeta},
\begin{align}
 &\prod_{l=1}^n\frac{(\prod_{i\notin\sigma}\Gamma(\alpha_{li}s_i))\Gamma(1-(A_\sigma^{-1}c)_l)}{\Gamma(1-(A_\sigma^{-1}c)_l+\sum_{i\notin\sigma}\alpha_{li}s_i)}\nonumber\\
=&\frac{1}{(1-e^{2\pi\sqrt{-1}(A_\sigma^{-1}c)_l})}\prod_{i\notin\sigma}\frac{1}{1-e^{-2\pi\sqrt{-1}\alpha_{li}s_i}}\int_{P_{\barsigma}}\prod_{i\notin\sigma}t_i^{\alpha_{li}s_i-1}(1-\sum_{i\notin\sigma}t_i)^{-(A_\sigma^{-1}c)_l}\prod_{i\notin\sigma}dt_i
\end{align}
Observe now that the variables $t_i$ are bounded from both below and above. Taking into account the integration contour, we have the following fundamental estimate on our integration contour:

\begin{equation}
\left|\prod_{l=1}^n\frac{(\prod_{i\notin\sigma}\Gamma(\alpha_{li}s_i))\Gamma(1-(A_\sigma^{-1}c)_l)}{\Gamma(1-(A_\sigma^{-1}c)_l+\sum_{i\notin\sigma}\alpha_{li}s_i)}\right|\leq C\cdot\prod_{i\notin\sigma}A_{li}^{|\alpha_{li}|\cdot |\Re(s_i)|},
\end{equation}

\noindent
where $C$ and $A_{li}$ are some constants. Now we would like to estimate the second term:
\begin{equation}
\prod_{i\notin\sigma}\frac{\Gamma(s_i)}{\prod_{l=1}^n\Gamma(\alpha_{li}s_i)}.
\end{equation}
Remember the famous Stirling's formula (c.f. \cite{KP}.)

\begin{equation}
\Gamma(s)\sim \sqrt{2\pi}e^{-s}s^{s-\frac{1}{2}} (s\rightarrow\infty,\; |arg(s)|<\pi).
\end{equation}
Since the inversion formula tells us that 
\begin{equation}
\Gamma(-s)=\frac{-\pi}{\Gamma(s+1)sin(\pi s)},
\end{equation}
the asymptotic behaviour of $\Gamma(-s)$ is
\begin{equation}
\Gamma(-s)\sim\Re(s)-(\Re(s)+\frac{1}{2})\log|s|+O(1)\;\;(-s\rightarrow\infty\text{ along our integration contour}).
\end{equation}
Therefore, the leading asymptotic behaviour is given by
\begin{equation}
\left|\prod_{i\notin\sigma}\frac{\Gamma(s_i)}{\prod_{l=1}^n\Gamma(\alpha_{li}s_i)}\right|\sim C\cdot |s_i|^{(1-\sum_l\alpha_{li})\Re(s_i)}e^{-(1-\sum_l\alpha_{li})\Re(s_i)}|s_i|^{power}\;\; (\Re(s_i)\rightarrow -\infty\text{ along our contour}).
\end{equation}

\noindent
When $\sum_l\alpha_{li}<1,$ this estimate implies it decays faster than any exponential.
On the other hand, when $\sum_l\alpha_{li}=1,$ the growth is at most of polynomial order.

Summing up, we have the following convergence result:

\begin{thm}
Put $H_\sigma=\{y\in\mathbb{R}^n;|A_\sigma^{-1}y|=1\}$ and \\
$U_\sigma=\{ z\in\mathbb{C}^N|\prod_{j\in\sigma}z_j\neq 0,|z_i|<R_i|z_\sigma^{A_\sigma^{-1}a(i)}|,\;\forall a(i)\in H_\sigma\setminus\sigma\}$, where $R_i$ is a small positive number. Then, $F(z)$ is convergent in $U_\sigma$.
\end{thm}

\section{Mellin-Barnes integral produces series solutions}\label{series}

\indent
In this section, we relate Mellin-Barnes integral to the so-called $\Gamma$ series. We follow the brilliant argument and notations of M.-C. Fern\'andez-Fern\'andez \cite{FF}. We renumber ${1,\cdots ,N}$ so that $\bar{\sigma}={1,\cdots,N-n}.$ Then, we can prove the following formula by a simple induction:

\begin{align}
  &F_{\s,0}(z)\\
=&\frac{z_\sigma^{-A_\sigma^{-1}c}}{(2\pi\sqrt{-1})^{N-n}}(\int_{(\Gamma_{0,1}-M)\times \Gamma_{0,1}\times\cdots\times \Gamma_{0,1}}+2\pi\sqrt{-1}\int_{L_M\times (\Gamma_{0,1}-M)\times \Gamma_{0,1}\times\cdots\times \Gamma_{0,1}}\\
 &+\cdots+(2\pi\sqrt{-1})^{N-n-1}\int_{L_M\times\cdots\times L_M\times (\Gamma_{0,1}-M)})\frac{\Gamma(s_{\overline{\sigma}})(e^{\pi\sqrt{-1}}z_{\overline{\sigma}})^{-s_{\overline{\sigma}}}z_\sigma^{A_\sigma^{-1}A_{\overline{\sigma}}s_{\overline{\sigma}}}}{\Gamma({\bf 1}-A_\sigma^{-1}(c-A_{\overline{\sigma}}s_{\overline{\sigma}}))}ds_{\overline{\sigma}}\\
 &+\displaystyle\sum_{m_1,\cdots, m_{N-n}=0}^{M-1}\frac{z_\sigma^{-A_\sigma^{-1}(c+A_{\overline{\sigma}}{\bf m})}z_{\overline{\sigma}}^{\bf m}}{\Gamma({\bf 1}-A_\sigma^{-1}(c+A_{\overline{\sigma}}{\bf m})){\bf m}!}.
\end{align}

\noindent
Here, $L_M$ is a positive loop around poles $0,-1,\cdots ,-M+1.$ When we fix $z\in U_\sigma$, we have 

\begin{equation}
|integrand|\leq C\prod_{i\in\overline{\sigma}}e^{-A_i|\Re(s_i)|}
\end{equation}

\noindent
uniformly on the integration contours if we are suitably away from spirals $\alpha_{li}^{-1}\mathbb{Z}.$ Here, $A_i$ is a positive constant. We can conclude that the reminder terms decay fast as $M$ tends to $\infty.$ Now let me introduce some notation following M.-C. Fern\'anez-Fern\'andez \cite{FF}. For any ${\bf k}\in\mathbb{Z}_{\geq 0}^{\overline{\sigma}}$, we put
\begin{equation}
\Lambda_{\bf k}=\{{\bf k+m}\in\mathbb{Z}_{\geq 0}^{\overline{\sigma}}|A_{\overline{\sigma}}{\bf m}\in\mathbb{Z}A_\sigma\}.
\end{equation}
We put $r=|\mathbb{Z}^{n\times 1}/\mathbb{Z}A_\sigma|.$ If we write $\mathbb{Z}^{n\times 1}/\mathbb{Z}A_\sigma=\{[A_{\overline{\sigma}}{\bf k}(j)]\}_{j=1}^r,$ we have a decomposition
\begin{equation}\label{disjoint}
\mathbb{Z}_{\geq 0}^{\overline{\sigma}}=\displaystyle \bigsqcup_{j=1}^{r}\Lambda_{{\bf k}(j)}.\end{equation}
We put
\begin{equation}
\varphi_{v^{\bf k}}(z)=z_\sigma^{-A_\sigma^{-1}c}\displaystyle\sum_{{\bf k+m}\in\Lambda_{\bf k}}\frac{z_\sigma^{-A_\sigma^{-1}A_{\overline{\sigma}}({\bf k+m})}z_{\overline{\sigma}}^{\bf k+m}}{\Gamma({\bf 1}-A_\sigma^{-1}(c+A_{\overline{\sigma}}({\bf k+m})))({\bf k+m})!}.
\end{equation}
By (\ref{disjoint}), we have 

\begin{prop}(\cite{FF})
Suppose $c$ is very generic. Then, $\{ \varphi_{v^{{\bf k}(j)}}(z)\}_{j=1}^r$ is a linearly independent set of solutions of $M_A(c).$
\end{prop}

\noindent
Now, we have the equality
\begin{equation}
F(z)=\displaystyle\sum_{j=1}^r\varphi_{v^{{\bf k}(j)}}(z).
\end{equation}

\noindent
Note that the convergence domain of $\varphi_{v^{{\bf k}(j)}}(z)$ is again $U_\sigma$. We are going to construct $r$ linearly independent Mellin-Barnes integral solutions. For any given $\tilde{\bf k}\in\mathbb{Z}^\sigma$, consider an analytic continuation $z_\sigma\mapsto e^{2\pi\sqrt{-1}\tilde{\bf k}}z_\sigma.$ We have the corresponding analytic continuation of $F_{\s,0}$ which is denoted by $F_{\s,\tilde{\bf k}}(z).$ This is concretely given by the formula
\begin{equation}
F_{\s,\tilde{\bf k}}(z)=\frac{e^{-2\pi\sqrt{-1}{}^t\tilde{\bf k}A_\sigma^{-1}c}z_\sigma^{-A_\sigma^{-1}c}}{(2\pi\sqrt{-1})^{N-n}}\int_{C_{\overline{\sigma}}}\frac{\Gamma(s_{\overline{\sigma}})(e^{\pi\sqrt{-1}}z_{\overline{\sigma}})^{-s_{\overline{\sigma}}}z_\sigma^{A_\sigma^{-1}A_{\overline{\sigma}}s_{\overline{\sigma}}}e^{2\pi\sqrt{-1}{}^t\tilde{\bf k}A_\sigma^{-1}A_{\overline{\sigma}}s_{\overline{\sigma}}}}{\Gamma({\bf 1}-A_\sigma^{-1}(c-A_{\overline{\sigma}}s_{\overline{\sigma}}))}ds_{\overline{\sigma}}.
\end{equation}

\noindent
A similar arguments as above leads to 

\begin{equation}
F_{\tilde{\bf k}}(z)=e^{-2\pi\sqrt{-1}{}^t\tilde{\bf k}A_\sigma^{-1}c}\displaystyle\sum_{j=1}^re^{-2\pi\sqrt{-1}{}^t\tilde{\bf k}A_\sigma^{-1}A_{\overline{\sigma}}{\bf k}(j)}\varphi_{v^{{\bf k}(j)}}(z).
\end{equation}

We show that the matrix $\left( e^{-2\pi\sqrt{-1}{}^t\tilde{\bf k}(i)A_\sigma^{-1}A_{\overline{\sigma}}{\bf k}(j)} \right)$ is invertible if $\{[\tilde{\bf k}(i)]\}_{i=1}^r$ denotes representatives of a finite abelian group $\mathbb{Z}^{n\times 1}/\mathbb{Z}{}^tA_\sigma$.
 This is due to the following elementary lemma.

\begin{lem}\label{pairing}
The pairing $<\; ,\;>:\mathbb{Z}^{n\times 1}/\mathbb{Z}{}^tA_\sigma\times\mathbb{Z}^{n\times 1}/\mathbb{Z}A_\sigma\rightarrow\mathbb{Q}/\mathbb{Z}$ defined by $<v,w>={}^tvA_\sigma^{-1}w$ is a non-degenerate pairing of finite abelian groups, i. e., for any fixed $v\in\mathbb{Z}^{n\times 1}/\mathbb{Z}{}^tA_\sigma,$ if one has $<v,w>=0$ for all $w\in\mathbb{Z}^{n\times 1}/\mathbb{Z}A_\sigma,$ then we have $v=0$ and vice versa.
\end{lem}
The proof is quite elementary because we know that there are invertible integer matrices $P,Q\in GL(n,\mathbb{Z})$ such that $P^{-1}A_\sigma Q=diag(1,\cdots ,1,d_1,\cdots ,d_l)$ where $d_1,\cdots ,d_l\in\mathbb{Z}^\times$ satisfy $d_1|d_2,\cdots ,d_{l-1}|d_l,$ and the lemma is obvious when $A_\sigma$ is replaced by $P^{-1}A_\sigma Q.$ Thanks to this lemma, we have the following

\begin{prop}
$\mathbb{Z}^{n\times 1}/\mathbb{Z}{}^tA_\sigma\tilde{\rightarrow}\widehat{\mathbb{Z}^{n\times 1}/\mathbb{Z}A_\sigma}$, where the isomorphism is induced from the pairing $<\; ,\;>.$

\end{prop}

\begin{proof}
Since we have a group embedding $\exp(-2\pi\ii\cdot):\mathbb{Q}/\mathbb{Z}\hookrightarrow\mathbb{C}^{*},$ the pairing $<\; ,\;>$ induces
\begin{equation}
\mathbb{Z}^{n\times 1}/\mathbb{Z}{}^tA_\sigma\tilde{\rightarrow}\Hom_{\mathbb{Z}}(\mathbb{Z}^{n\times 1}/\mathbb{Z}A_\sigma,\mathbb{Q}/\mathbb{Z})\hookrightarrow\widehat{\mathbb{Z}^{n\times 1}/\mathbb{Z}A_\sigma}.
\end{equation}
Counting the number of elements, we have the proposition.
\end{proof}

Thus, the matrix $\frac{1}{r}\left( e^{-2\pi\sqrt{-1}{}^t\tilde{\bf k}(i)A_\sigma^{-1}A_{\overline{\sigma}}{\bf k}(j)} \right)$ is a unitary matrix by the orthogonality of irreducible characters.

\begin{rem}
For any finite abelian group $G$, we have an isomorphism 
\begin{equation}
\Hom (G,\mathbb{Q}/\mathbb{Z})\simeq\Ext^1(G,\mathbb{Z}).
\end{equation}
The pairing of Lemma \ref{pairing} is induced naturally from this isomorphism.
\end{rem}
\noindent
Employing the terminology of regular triangulation (\cite{FF}), we obtain the main

\begin{thm}\label{thm:MainTheoremForMellin-Barnes}
Let $T$ be a regular triangulation such that $\forall\s\in T$, $|A_\s^{-1}{\bf a}(j)|\leq 1(\forall j\in\bs)$. We assume $c$ is very generic. Then, for each element $\tilde{\bf k}\in\Z^{n\times 1}/\Z{}^tA_\s$, we can associate an analytic continuation $F_{\s,\tk}$ of $F_{\s,0}$. If $\{\tk(i)\}_{i=1}^r$ denotes a complete system of representatives of $\Z/\Z{}^tA_\s$, the set of functions $\bigcup_{\s\in T}\{ F_{\s,\tk(i)}\}_{i=1}^r$ forms a basis of solutions of $M_A(c)$. Moreover, the transformation matrix between $\{F_{\s,\tk(i)}\}_{i=1}^r$ and $\{\varphi_{\sigma,{\bf k}(i)}\}_{i=1}^r$ is given by 
\begin{equation}
T_\sigma=\diag\left(\exp\{-2\pi\sqrt{-1}{}^t\tilde{\bf k}(i)A_\sigma^{-1}c\}\right)_{i=1}^r\Bigl(\exp\{ -2\pi\sqrt{-1}{}^t\tilde{\bf k}(i)A^{-1}_\sigma A_{\bar{\sigma}}{\bf k}(j)\}\Bigr)_{i,j=1}^r
\end{equation}
i.e., we have a relation
\begin{equation}
\begin{pmatrix}
F_{\s,\tk(1)}\\
\vdots\\
F_{\s,\tk(r)}
\end{pmatrix}
=
T_\s
\begin{pmatrix}
\varphi_{\sigma,{\bf k}(1)}\\
\vdots\\
\varphi_{\sigma,{\bf k}(r)}
\end{pmatrix}.
\end{equation}
Here, the second matrix $\Bigl(\exp\{ -2\pi\sqrt{-1}{}^t\tilde{\bf k}(i)A^{-1}_\sigma A_{\bar{\sigma}}{\bf k}(j)\}\Bigr)_{i,j=1}^r$ is the character matrix of $\Z^{n\times 1}/\Z A_\s$. 
\end{thm}

\end{document}